\documentclass[11pt]{article}
\usepackage{amsmath,amsfonts,amssymb,amsthm,fancyhdr,bm,mathtools,enumitem}
\usepackage[hidelinks]{hyperref}
\usepackage{fullpage}
\usepackage[dvipsnames]{xcolor}
\usepackage[capitalize]{cleveref}
\usepackage[color=Purple!50!white]{todonotes}
\usepackage[numbers,sort]{natbib}

\newtheorem{thm}{Theorem}[section]
\newtheorem{lem}[thm]{Lemma}
\newtheorem{prop}[thm]{Proposition}

\newtheorem{conj}[thm]{Conjecture}

\theoremstyle{definition}

\crefname{equation}{equation}{equations}
\crefname{lem}{Lemma}{Lemmas}
\crefname{prop}{Proposition}{Propositions}
\crefname{claim}{Claim}{Claims}
\crefname{thm}{Theorem}{Theorems}
\crefname{conj}{Conjecture}{Conjectures}

\newlist{lemenum}{enumerate}{1}
\setlist[lemenum]{label=(\alph*), ref=\thelem(\alph*)}
\crefalias{lemenumi}{lemma}

\newcommand\ab[1]{\lvert#1\rvert}

\newcommand{\flo}[1]{\lfloor #1 \rfloor}

\newcommand\R{\mathbb{R}}

\title{Ramsey numbers upon vertex deletion}
\author{Yuval Wigderson\thanks{Institute for Theoretical Studies, ETH Z\"urich, 8092 Z\"urich, Switzerland. Email address: \texttt{yuval.wigderson@eth-its.ethz.ch}. Supported in part by NSF GRFP Grant DGE-1656518, ERC Consolidator Grants 863438 and 101044123, and NSF-BSF Grant 20196.}}
\date{}

\begin{document}
\maketitle
\begin{abstract}
	Given a graph $G$, its Ramsey number $r(G)$ is the minimum $N$ so that every two-coloring of $E(K_N)$ contains a monochromatic copy of $G$. It was conjectured by Conlon, Fox, and Sudakov that if one deletes a single vertex from $G$, the Ramsey number can change by at most a constant factor. We disprove this conjecture, exhibiting an infinite family of graphs such that deleting a single vertex from each decreases the Ramsey number by a super-constant factor.

	One consequence of this result is the following. There exists a family of graphs $\{G_n\}$ so that in any Ramsey coloring for $G_n$ (that is, a coloring of a clique on $r(G_n)-1$ vertices with no monochromatic copy of $G_n$), one of the color classes has density $o(1)$.
\end{abstract}
\section{Introduction}
The Ramsey number of a graph $G$ is the minimum $N$ so that every two-coloring of the edges of the complete graph $K_N$ contains a monochromatic copy of $G$.
Motivated by understanding the concentration of the Ramsey number of the Erd\H os--R\'enyi random graph $G(n,p)$, Conlon, Fox, and Sudakov \cite{MR4170438} made the following natural conjecture.
\begin{conj}[{\cite[Conjecture 5.1]{MR4170438}}]\label{conj:cfs}
	There exists a universal constant $C>0$ so that the following holds. Let $G$ be a graph, and let $H$ be obtained from $G$ by deleting a single vertex. Then
	\[
		r(G) \leq C\cdot r(H).
	\]
\end{conj}
Conlon, Fox, and Sudakov proved two natural weakenings of this conjecture. First, they showed that \cref{conj:cfs} holds when $G$ is a dense graph, i.e.\ when it has $n$ vertices and at least $p\binom n2$ edges, and $C$ is allowed to depend on $p$. They also proved a slightly weaker inequality, namely that for any $(n+1)$-vertex graph $G$ and any graph $H$ obtained by deleting a single vertex from $G$, one has
\begin{equation}\label{eq:cfs-bound}
	r(G) \leq 2 n\cdot r(H).
\end{equation}

Our goal in this paper is to disprove \cref{conj:cfs}. We find an explicit family of graphs such that deleting a single vertex decreases the Ramsey number by a super-constant factor. Here, and in the rest of the paper, all logarithms have base $2$.
\begin{thm}\label{thm:main}
	For any $n\geq 16$, there exists an $(n+1)$-vertex graph $G$ with Ramsey number $r(G) > \frac 13 n\log n$. However, there is a vertex of $G$ whose deletion yields a graph $H$ with Ramsey number $r(H) = n$. In particular, $r(G) = \omega(r(H))$.
\end{thm}
In the other direction, we prove the following strengthening of (\ref{eq:cfs-bound}).
\begin{thm}\label{thm:sqrt-ub}
	Let $G$ be an $(n+1)$-vertex graph, and suppose that $H$ is obtained from $G$ by deleting one vertex. Then 
	\[
		r(G) \leq C\sqrt{n \log n}\cdot r(H)
	\]
	for some absolute constant $C>0$. 
\end{thm}

Recall that the edge density of an $N$-vertex graph is its number of edges divided by $\binom N2$. For $\varepsilon \in (0,\frac 12]$, let us say that a two-coloring of $E(K_N)$ is \emph{$\varepsilon$-balanced} if both color classes have edge density at least $\varepsilon$. Ramsey properties of $\varepsilon$-balanced colorings have been well studied, see e.g.\ \cite{MR325446,MR2455594,MR2433767,MR4267031,MR4625871}.
A \emph{Ramsey coloring} for a graph $G$ is a two-coloring of $E(K_N)$, where $N=r(G)-1$, which contains no monochromatic copy of $G$. In general, there may be multiple non-isomorphic Ramsey colorings for $G$. For $\varepsilon \in (0,\frac 12]$, let us say that $G$ is \emph{$\varepsilon$-Ramsey-balanced} if there exists a Ramsey coloring for $G$ which is $\varepsilon$-balanced.

Our next result shows that for Ramsey-balanced graphs $G$, \cref{conj:cfs} is true.
\begin{prop}\label{prop:ramsey-balanced}
	Let $\varepsilon \in (0,\frac 12]$, let $G$ be an $\varepsilon$-Ramsey-balanced graph, and let $H$ be obtained from $G$ by deleting one vertex. Then
	\[
		r(G) \leq \frac{4}{\varepsilon}\cdot r(H).
	\]
\end{prop}
\cref{prop:ramsey-balanced} shows that \cref{conj:cfs} is true for any graph $G$ which is $\Omega(1)$-Ramsey-balanced. Combined with \cref{thm:main}, we conclude that for any fixed $\varepsilon>0$, there exist graphs which are \emph{not} $\varepsilon$-Ramsey-balanced. This is somewhat surprising, as one might naively expect all graphs to be Ramsey-balanced: since we are searching for the same graph in both the red and blue colors, it seems plausible that any extremal coloring should have roughly equal red and blue densities, or at least that these densities would be within a constant factor of one another.

The rest of this paper is organized as follows. In \cref{sec:main-proof}, we exhibit explicit graphs in which the removal of one vertex changes the Ramsey number by a super-constant factor, proving \cref{thm:main}. In \cref{sec:upper-bounds}, we prove \cref{thm:sqrt-ub}, establishing a stronger upper bound than (\ref{eq:cfs-bound}). In \cref{sec:balanced}, we prove \cref{prop:ramsey-balanced}, showing that \cref{conj:cfs} is true for Ramsey-balanced graphs. In \cref{sec:multicolor}, we make some comments on the multicolor version of these problems. We end with some concluding remarks and open problems in \cref{sec:conclusion}. For clarity of presentation, we systematically omit floor and ceiling signs whenever they are not crucial.

\section{Proof of Theorem \ref{thm:main}}\label{sec:main-proof}
Let $k$ and $n \geq 4^k$ be positive integers. Let $H_{k,n}$ be the $n$-vertex graph consisting of the clique $K_k$, plus $n -k$ isolated vertices. Additionally, let $G_{k,n}$ be obtained from $H_{k,n}$ by adding an apex vertex, i.e.\ a vertex adjacent to all vertices of $H_{k,n}$. We note for future reference that $G_{k,n}$ is connected and has chromatic number $k+1$. \cref{thm:main} follows from the following two lemmas, both of which have simple proofs using standard techniques, which we give after stating the lemmas.
\begin{lem}\label{lem:H-ub}
	$r(H_{k,n})=n$ if $n \geq 4^k$.
\end{lem}
\begin{lem}\label{lem:G-lb}
	$r(G_{k,n}) > nk$.
\end{lem}
In particular, given $n \geq 16$, let $k = \flo{\frac 12 \log n}$, so that $4^k \leq n < 4^{k+1}$. Then $r(G_{k,n}) > nk \geq \frac 13 n \log n$ since $n \geq 16$. This shows that \cref{thm:main} follows from \cref{lem:H-ub,lem:G-lb}.

\begin{proof}[Proof of \cref{lem:H-ub}]
	The lower bound $r(H_{k,n})\geq n$ is immediate since every $n$-vertex graph has Ramsey number at least $n$. For the upper bound, fix a two-coloring of $E(K_n)$. Erd\H os and Szekeres~\cite{MR1556929} proved that $r(K_k) \leq 4^k \leq n$, so this coloring contains a monochromatic $K_k$. Together with the remaining $n-k$ vertices, we obtain a monochromatic copy of $H_{k,n}$.
\end{proof}
\begin{proof}[Proof of \cref{lem:G-lb}]
	Let $N = nk$, and consider the \emph{Tur\'an coloring} of $E(K_N)$. Namely, we equitably partition $V(K_N)$ into sets $S_1,\dots,S_k$, each of order $n$. We color all edges inside some $S_i$ red, and all edges between $S_i$ and $S_j$ blue, for all $i \neq j$. Since $G_{k,n}$ is connected and has $n+1$ vertices, there can be no monochromatic red copy of $G_{k,n}$, as the connected components of the red graph in this coloring each have $n$ vertices. The blue graph, on the other hand, has chromatic number $k$, and $G_{k,n}$ has chromatic number $k+1$, so there can be no monochromatic blue copy of $G_{k,n}$ either. This shows that $r(G_{k,n})>nk$, as claimed.
\end{proof}

\section{Upper bounds}\label{sec:upper-bounds}

A graph $H$ is called \emph{$d$-degenerate} if every subgraph of $H$ has a vertex of degree at most $d$. Equivalently, $H$ is $d$-degenerate if one can order its vertices so that every vertex has at most $d$ neighbors which precede it in the ordering. The \emph{degeneracy} of $H$ is defined as the minimum $d$ so that $H$ is $d$-degenerate. 

\cref{thm:sqrt-ub} follows from the following result, which gives two bounds on $r(G)/r(H)$ when $H$ is obtained from $G$ by deleting a vertex; the first is strong when $H$ is sparse (i.e.\ has low degeneracy), while the second is stronger when $H$ is dense (i.e.\ has high degeneracy). Balancing the two bounds, we obtain \cref{thm:sqrt-ub}.

\begin{prop}\label{prop:degen-balance}
	Let $H$ be obtained from $G$ by deleting a vertex. If $H$ has $n$ vertices and degeneracy $d\geq 1$, then
	\begin{equation}\label{eq:first-bd}
		r(G) \leq 4 d r(H),
	\end{equation}
	and
	\begin{equation}\label{eq:second-bd}
		r(G)\leq \left(A \frac nd \log \frac nd\right) r(H)
	\end{equation}
	where $A>0$ is an absolute constant.
\end{prop}
We remark that the graph $H_{k,n}$ constructed in \cref{sec:main-proof} has degeneracy $k = \Theta (\log n)$. Since $r(G_{k,n})/r(H_{k,n}) = \Omega(\log n)$, this shows that \eqref{eq:first-bd} is tight up to the constant factor for the graphs $G_{k,n}$ and $H_{k,n}$. This implies that if one were to search for stronger counterexamples to \cref{conj:cfs}, they would need to have somewhat higher degeneracy (i.e.\ $d =\omega(\log n)$).

Assuming \cref{prop:degen-balance}, \cref{thm:sqrt-ub} follows.
\begin{proof}[Proof of \cref{thm:sqrt-ub}]
	Note that for $d \geq 1$, we have
	\[
		\min \left\{ 4d, A \frac nd \log \frac nd \right\} \leq \sqrt{4d \cdot A \frac nd \log \frac nd} = \sqrt{4An \log \frac nd} \leq \left(2 \sqrt A\right) \sqrt{n \log n}.
	\]
	Setting $C=2\sqrt A$ and applying \cref{prop:degen-balance} yields \cref{thm:sqrt-ub}.
\end{proof}

To prove \cref{prop:degen-balance}, we need the following simple and well-known lemma, which shows that one can embed a graph of bounded degeneracy in a very dense graph.
\begin{lem}\label{lem:greedy}
	Let $H$ be an $n$-vertex $d$-degenerate graph, and let $R$ be an $N$-vertex graph in which every vertex has at most $(N-n)/d$ non-neighbors. Then $H$ is a subgraph of $R$.
\end{lem}
\begin{proof}
	Let the vertices of $H$ be $v_1,\dots,v_n$, sorted so that each $v_i$ has at most $d$ neighbors $v_j$ with $j<i$. We inductively pick an embedding $\varphi:V(H) \to V(R)$, maintaining the property that $\varphi(v_1),\dots,\varphi(v_i)$ span a copy of $H[v_1,\dots,v_i]$. First, we let $\varphi(v_1)$ be an arbitrary vertex of $R$. Suppose we've defined $\varphi(v_1),\dots,\varphi(v_{i-1})$, and we wish to define $\varphi(v_i)$. Let $J$ be the set of $j<i$ with $v_j v_i \in E(H)$, so that $\ab J \leq d$. For every $j \in J$, there are at most $(N-n)/d$ non-neighbors of $\varphi(v_j)$ in $R$, and thus at most $N-n$ vertices of $R$ are non-adjacent to $\varphi(v_j)$ for some $j \in J$. Thus, there are at least $n$ vertices of $R$ that are adjacent to all $\{\varphi(v_j)\}_{j \in J}$, and fewer than $n$ of these vertices have been used in embedding $v_1,\dots,v_{i-1}$. So we pick any of the remaining candidate vertices as $\varphi(v_i)$, which maintains our inductive hypothesis. Continuing in this manner until $i=n$, we find a copy of $H$ in $R$.
\end{proof}

We will use the following result of Erd\H os and Szemer\'edi \cite{MR325446}. It shows that colorings which are not $\varepsilon$-balanced contain larger cliques than what is implied by the Ramsey number alone.
\begin{lem}[{\cite[Theorem 2]{MR325446}}]\label{lem:erdos-szemeredi}
	There exists an absolute constant $a>0$ such that the following holds for all $\varepsilon \in (0,\frac 12]$ and every positive integer $N$. Any two-coloring of $E(K_N)$ which is not $\varepsilon$-balanced contains a monochromatic clique of order $\frac{a}{\varepsilon \log \frac 1 \varepsilon} \log N$.
\end{lem}

Finally, we will need the following well-known lower bound on the Ramsey number of a graph of high degeneracy (see e.g.\ \cite[Section 11, Remark 2]{MR2520279}). We provide the proof for completeness.
\begin{lem}\label{lem:exp-degen}
	If $G$ is a graph with degeneracy $d\geq 1$, then $r(G)\geq {2^{d/2}}$. 
\end{lem}
\begin{proof}
	By the definition of degeneracy, there exists a subgraph $G_0 \subseteq G$ with minimum degree at least $d$. Let $G_0$ have $k$ vertices, so that it has at least $dk/2$ edges, and note that $k \geq 2$ since $d \geq 1$. Let $N={2^{d/2}}$, and consider a uniformly random two-coloring of $E(K_N)$. The expected number of monochromatic copies of $G_0$ is at most
	\[
		k! \binom Nk 2^{1-dk/2} < N^k 2^{1-dk/2} = 2 \left(\frac{N}{2^{d/2}}\right)^k = 2.
	\]
	Thus, there exists a two-coloring of $E(K_N)$ with fewer than $2$ monochromatic copies of $G_0$. By deleting one vertex, we obtain a coloring on $N-1$ vertices with no monochromatic copies of $G_0$, and thus, no monochromatic copies of $G$. This shows that $r(G) \geq (N-1)+1=N$, as claimed. 
\end{proof}

We are now ready to prove \cref{prop:degen-balance}.
\begin{proof}[Proof of \cref{prop:degen-balance}]
	Let $N=r(G)-1$, and fix a two-coloring of $E(K_N)$ with no monochromatic copy of $G$. Let $v$ be a vertex of $K_N$. Assume without loss of generality that $v$ has at least $(N-1)/2$ red neighbors, and let $S$ be this red neighborhood. Now consider the induced coloring on $S$. Let $w$ be a vertex of maximum blue degree in this induced coloring, and let $T$ be the blue neighborhood of $w$ in $S$. 
	Then every vertex in $T$ is adjacent to $v$ in red and to $w$ in blue, and thus $T$ cannot contain a monochromatic copy of $H$. Therefore, $\ab T \leq r(H)-1$. We now lower-bound $\ab T$ in two different ways.

	First, we claim that $\ab T > (\ab S-n)/d$. Since $w$ is a vertex of maximum blue degree, we see that if this is not the case, then every vertex in $S$ has fewer than $(\ab S-n)/d$ blue neighbors in $S$. Applying \cref{lem:greedy} with $R$ being the red graph on $S$, we find that there is a red copy of $H$ in $S$. Together with $v$, this yields a red copy of $G$, a contradiction. 
	This proves the claim, which implies that
	\[
		r(H) -1 \geq \ab T > \frac{\ab S-n}{d} \geq \frac{N-1}{2d} - \frac nd 
	\]
	and thus
	\[
		N < 2dr(H) - 2d +2n+1 \leq 2dr(H) + 2n -1 \leq 4dr(H)
	\]
	since $d\geq 1$ and since $n \leq r(H)$. Recalling that $N = r(G)-1$ yields \eqref{eq:first-bd}.

	Note that \eqref{eq:second-bd} follows from (\ref{eq:cfs-bound}) if $d \leq 9$ by choosing $A$ sufficiently large, so we henceforth assume $d \geq 10$.
	We now claim that $\ab T \geq \varepsilon (\ab S-1)$, where $\varepsilon=cd/(n\log \frac nd)$, for some sufficiently small constant $c>0$. If this is not the case, then every vertex in $S$ has blue degree less than $\varepsilon (\ab S-1)$, and thus the blue edge density in $S$ is less than $\varepsilon$. So the induced coloring on $S$ is not $\varepsilon$-balanced, and by \cref{lem:erdos-szemeredi}, we may find in $S$ a monochromatic clique of order $\frac{a}{\varepsilon \log \frac 1 \varepsilon}\log {\ab S}$. Note that, as $H \subseteq G$, the degeneracy of $G$ is at least $d$, and therefore \cref{lem:exp-degen} implies that
	\[
		\ab S \geq \frac{N-1}{2} = \frac{r(G)-2}{2} \geq \frac{2^{d/2}-2}{2} \geq 2^{d/3},
	\]
	by our assumption that $d \geq 10$. Note too that by choosing $c$ sufficiently small, we can ensure that
	\[
		\varepsilon \log \frac 1 \varepsilon = \frac{cd}{n\log \frac nd} \cdot \left(\log \frac nd + \log \frac 1c + \log \log \frac nd\right) \leq \frac{ad}{6n}.
	\]
	Therefore,
	\[
		\frac{a}{\varepsilon \log \frac 1 \varepsilon} \log {\ab S} \geq a \cdot \frac{6n}{ad} \cdot \frac d3 = 2n.
	\]
	This shows that $S$ contains a monochromatic clique of order $2n$, and thus a copy of $G$, a contradiction. This proves the claim that $\ab T \geq \varepsilon(\ab S-1)$, which implies that
	\[
		r(H)-1 \geq \ab T \geq \varepsilon (\ab S-1) \geq \frac{\varepsilon N}{4}
	\]
	and thus
	\[
		r(G) = N+1 \leq \frac 4 \varepsilon r(H) = \left(A \frac nd \log \frac nd\right) r(H)
	\]
	for a constant $A = 4/c>0$.
\end{proof}

\section{Ramsey-balanced graphs}\label{sec:balanced}
In this section we discuss Ramsey-balanced graphs and prove \cref{prop:ramsey-balanced}, which says that \cref{conj:cfs} is true for Ramsey-balanced graphs. As discussed in the Introduction, this shows that there exists a sequence of graphs $G_n$ such that in any Ramsey coloring for $G_n$, one of the color classes has density $o(1)$.

For a vertex $v$ in a two-colored complete graph, we denote by $N_R(v), N_B(v)$ the red and blue neighborhoods of $v$, respectively. The only property we need of $\varepsilon$-balanced colorings is that in any $\varepsilon$-balanced coloring, we can find vertices $v,w$ such that ${N_R(v) \cap N_B(w)}$ is large. Such a statement can be easily deduced from \cite[Lemma 2.1]{MR2455594}; however, the following result, due to Kam\v cev and M\"uyesser\footnote{This result is mentioned without proof in \cite{MR4625871}. The proof given here is due to Kam\v cev and M\"uyesser, and is included with their kind permission.}, gives a stronger quantitative bound.

\begin{lem}[Kam\v cev--M\"uyesser]\label{cor:common-nbhd}
	Let $0<\varepsilon \leq \frac 12$, and fix an $\varepsilon$-balanced coloring of $E(K_N)$ with colors red and blue. Then there exist $v,w \in V(K_N)$ such that $\ab{N_R(v)\cap N_B(w)} \geq \frac{\varepsilon}{4} (N-1)$.
\end{lem}
In the course of the proof of \cref{cor:common-nbhd}, we will need the following simple fact about convex functions, which can be viewed as a ``reversed'' form of Jensen's inequality. Although this result appears to be well-known, we were not able to find a concise statement in the literature, and include the short proof for a completeness.
\begin{lem}[Reverse Jensen inequality]\label{lem:reverse jensen}
	Let $f:[0,s] \to \R$ be a convex function. Let $x_1,\dots,x_N \in [0,s]$ be real numbers, and let $x=\sum_i x_i$. Then
	\[
		\sum_{i=1}^N f(x_i) \leq \frac{Ns-x}{s} f(0) + \frac xs f(s).
	\]
\end{lem}
\begin{proof}
	The definition of convexity implies that
	\[
		f(x_i) \leq \frac{s-x_i}{s}f(0) + \frac{x_i}{s} f(s)
	\]
	for every $1 \leq i \leq N$. Summing this up over all $i$ gives the claimed result.
\end{proof}
We are now ready to prove \cref{cor:common-nbhd}.
\begin{proof}[Proof of \cref{cor:common-nbhd}]
	By swapping the roles of the colors if necessary, we may assume without loss of generality that at least half the edges are colored blue. 
	We may also assume that exactly $\varepsilon \binom N2$ edges are red, and the remainder are blue; if this is not the case, then we simply replace $\varepsilon$ by a larger value, and obtain a stronger conclusion. Thus, the average red degree is $\varepsilon(N-1)$, and the average blue degree is $(1-\varepsilon)(N-1)$.

	Suppose first that there is some vertex $v$ with $\ab{N_R(v)} \geq \frac 32 \varepsilon(N-1)$. Let $w$ be a vertex of maximum blue degree, so that $\ab{N_B(w)} \geq (1-\varepsilon)(N-1)$. Then by the pigeonhole principle, we have that 
	\[
		\ab{N_R(v) \cap N_B(w)} \geq \frac \varepsilon 2 (N-1).
	\]
	Therefore we may assume that every vertex of $K_N$ has red degree at most $\frac 32 \varepsilon (N-1)$. Applying \cref{lem:reverse jensen} with $f(y)=y^2$ and $s=\frac 32 \varepsilon(N-1)$, this implies that 
	\[
		\sum_{u \in V(K_N)} \ab{N_R(u)}^2 \leq \frac{\sum_u \ab{N_R(u)}}{\frac 32 \varepsilon (N-1)}\left( \frac 32 \varepsilon (N-1)\right)^2
		= \frac{2N}{3} \left( \frac 32 \varepsilon (N-1)\right)^2 
		\leq \frac 34 \varepsilon N(N-1)^2,
	\]
	where the final inequality uses that $\varepsilon \leq \frac 12$.
	Let $X$ denote the number of copies of $K_{1,2}$ with one red and one blue edge. We have that
	\begin{align*}
		X &= \sum_{u \in V(K_N)} \ab{N_R(u)} (N-1-\ab{N_R(u)})\\&= (N-1) \sum_{u \in V(K_N)} \ab{N_R(u)} - \sum_{u \in V(K_N)} \ab{N_R(u)}^2 \\
		&\geq (N-1) \cdot \varepsilon N(N-1) - \frac 34 \varepsilon N(N-1)^2\\
		&=\frac \varepsilon 4  N(N-1)^2.
	\end{align*}
	On the other hand, we have that
	\[
		X = \sum_{v,w \in V(K_N)} \ab{N_R(v) \cap N_B(w)} = \sum_{\substack{v,w \in V(K_N)\\v \neq w}} \ab{N_R(v) \cap N_B(w)},
	\]
	where the second equality is because $N_R(v) \cap N_B(v)=\varnothing$. This implies that there exist $v,w$ such that
	\[
		\ab{N_R(v) \cap N_B(w)} \geq \frac{X}{N(N-1)} \geq \frac \varepsilon 4 (N-1).\qedhere
	\]
\end{proof}

With these preliminaries, \cref{prop:ramsey-balanced} is an easy consequence.

\begin{proof}[Proof of \cref{prop:ramsey-balanced}]
	Let $N=r(G)-1$, and fix an $\varepsilon$-balanced coloring of $E(K_N)$ with no monochromatic copy of $G$. \cref{cor:common-nbhd} yields two vertices $v,w \in V(K_N)$ with $\ab{N_R(v) \cap N_B(w)} \geq \frac{\varepsilon}{4}(N-1)$. Let $S= N_R(v) \cap N_B(w)$. We claim that $S$ contains no monochromatic copy of $H$. Indeed, if there were some red copy of $H$ in $S$, then by adding $v$ to it we would find a red copy of $G$; similarly, a blue copy of $H$ in $S$ yields a blue copy of $G$ by adding $w$. So we conclude that $\ab S \leq r(H)-1$, which implies that
	\[
		r(G) = N+1 \leq \frac{4}{\varepsilon} \ab S+2 \leq \frac{4}{\varepsilon} (r(H)-1)+2 \leq \frac{4}{\varepsilon} r(H).\qedhere
	\]
\end{proof}

To conclude this section, we remark that the results of \cref{sec:upper-bounds} can be more or less equivalently phrased in terms of Ramsey-balanced graphs. Namely, the proof of \cref{prop:degen-balance} can be used to show that if $H$ is an $n$-vertex graph of degeneracy $d\geq 1$, then $H$ is both $\Omega(\frac1d)$-Ramsey-balanced and $\Omega(d/(n\log \frac nd))$-Ramsey-balanced. Balancing these two bounds and plugging them into \cref{prop:ramsey-balanced} yields \cref{thm:sqrt-ub}.

\section{More colors}\label{sec:multicolor}
For an integer $q \geq 2$, let $r(H;q)$ denote the $q$-color Ramsey number of $H$, that is, the least $N$ so that every $q$-coloring of $E(K_N)$ contains a monochromatic copy of $H$. When the number of colors is greater than $2$, we can prove a stronger version of \cref{thm:main}, showing a polynomial gap between $r(G)$ and $r(H)$.

\begin{thm}\label{thm:multicolor}
	Fix an integer $q \geq 3$. For all sufficiently large $n$, there exists an $(n+1)$-vertex graph $G$ with Ramsey number $r(G;q)>n^{1+ \frac{3q-5}{8q\log q}-o(1)}$. However, there is a vertex of $G$ whose deletion yields a graph $H$ with Ramsey number $r(H;q) =n$. In particular, for sufficiently large $n$,
	\[
		r(G;q) \geq (r(H;q))^{1+\alpha}
	\]
	for some fixed $\alpha>0$ depending only on $q$.
\end{thm}
\begin{proof}
	Fix positive integers $k$ and $n\geq q^{qk}$, and let $H_{k,n}$ and $G_{k,n}$ be the graphs defined in \cref{sec:main-proof}.
	As in the proof of \cref{lem:H-ub}, it is easy to check that $r(H_{k,n};q) = n$, by using the well-known upper bound $r(K_k;q) \leq q^{qk}$.

	For the lower bound on $r(G_{k,n};q)$, let $m=r(K_{k+1};q-1)-1$ and $N=mn$. By the definition of $r(K_{k+1};q-1)$, there exists a coloring $\chi:E(K_m) \to [q-1]$ with no monochromatic copy of $K_{k+1}$. We partition $V(K_N)$ into $m$ blocks $S_1,\dots,S_m$, each comprising $n$ vertices. For all $i \neq j$, we color all edges between $S_i$ and $S_j$ by the color $\chi(i,j) \in [q-1]$, and we color all edges inside some part with color $q$. There is no copy of $K_{k+1}$ in any of the first $q-1$ colors, and thus no copy of $G_{k,n}$ in any of these colors. Additionally, $G_{k,n}$ is connected and has $n+1$ vertices, so there is no copy of $G_{k,n}$ in the $q$th color either. This shows that $r(G_{k,n};q)>N$.

	To conclude, we need a lower bound on $r(K_{k+1};q-1)$ that grows exponentially in $q$. Such a bound was first proved by Lefmann \cite{MR932230}, with recent improvements by Conlon--Ferber \cite{MR4186575}, the author~\cite{MR4246789}, and Sawin \cite{MR4357352}. For concreteness, we quote the bound from \cite{MR4246789}, which says that for fixed $q\geq 3$ and $k \to \infty$,
	\[
		r(K_{k+1};q-1) > 2^{\frac{3q-5}{8}k - o(k)}.
	\]
	Letting $k = \flo{\log n/(q \log q)}$, we conclude that
	\[
		r(G_{k,n};q) > (r(K_{k+1};q-1)-1)n \geq n^{1+ \frac{3q-5}{8q\log q}-o(1)}.\qedhere
	\]
\end{proof}
We remark that one can get a slightly stronger bound on $\alpha$ for $q \geq 4$ by using the main result from \cite{MR4357352}, which says that
\[
	r(K_{k+1};q-1) > 2^{(0.38796(q-3)+\frac 12)k-o(k)}.
\]
However, as the value of $\alpha$ we obtain is likely far from optimal in any case, we chose to use the somewhat simpler expression from \cite{MR4246789}.

One strange feature of \cref{thm:multicolor} is that the bound on $\alpha$ actually deteriorates as $q$ gets larger, since the expression $(3q-5)/(8q\log q)$ tends to $0$ as $q \to \infty$. However, this is really a consequence of our poor understanding of multicolor Ramsey numbers. Indeed, the best known bounds on $r(K_k;q)$ are
\[
	{cqk} \leq \log r(K_k;q) \leq {qk}\log q
\]
for some constant $c>0$. The logarithmic gap between the lower and upper bounds here appears as the factor of $\log q$ in the denominator of $\alpha$. As such, we expect that a better understanding of the asymptotics of $r(K_k;q)$ would lead to a bound on $\alpha$ which does not deteriorate as $q \to \infty$.

While the lower bound of \cref{thm:multicolor} is noticeably stronger than that of \cref{thm:main}---we obtain a polynomial rather than barely super-linear bound when $q\geq 3$---unfortunately, the upper bounds in the multicolor case are much worse. Indeed, mimicking the proof\footnote{The key observation is that, in a $q$-coloring of $K_N$, we may assume that the density of every color is at least $1/r(G;q-1)$, as otherwise we may apply Tur\'an's theorem to find a set of $r(G;q-1)$ vertices colored by only $q-1$ colors, which then contains a monochromatic copy of $G$. Applying this observation $q$ times in nested common neighborhoods, we can find a set of $N/r(G;q-1)^q$ vertices, which is complete in each of the $q$ colors to some other vertex. Hence $r(G;q)/r(G;q-1)^q < r(H;q)$, which implies the claimed bound since $r(G;q-1)^q \leq 2^{c_qn}$.} of (\ref{eq:cfs-bound}), one obtains that if $G$ is an $n$-vertex graph and $H$ is obtained from $G$ by deleting a single vertex, then
\[
	r(G;q) \leq 2^{c_q n} r(H;q)
\]
for an absolute constant $c_q>0$ depending only on $q$, for any fixed $q \geq 3$. In other words, in contrast to the $O(n)$ bound on $r(G)/r(H)$ in (\ref{eq:cfs-bound}), in the case of $q\geq 3$, the bound on $r(G;q)/r(H;q)$ is exponential in $n$. Given that $r(G;q)$ is itself at most exponential in $n$, this bound is extremely weak, and it is not clear how to meaningfully improve it.

\section{Concluding remarks}\label{sec:conclusion}
Although \cref{thm:main} disproves the original conjecture of Conlon, Fox, and Sudakov, there remain a number of other interesting open problems.

The most natural question is to close the gap between \cref{thm:main,thm:sqrt-ub}. Formally, we can define
\[
	f(n) \coloneqq \max \left\{\frac{r(G)}{r(G \setminus \{v\})} : G\text{ is an }(n+1)\text{-vertex graph and }v \in V(G)\right\},
\]
which measures how much the deletion of a single vertex can affect the Ramsey number of an $(n+1)$-vertex graph. Then \cref{thm:main,thm:sqrt-ub} imply that $\Omega(\log n) \leq f(n) \leq O(\sqrt{n\log n})$. It would be interesting to improve either bound; concretely, we make the following conjecture.
\begin{conj}
	There is some $c>0$ so that $f(n) \geq n^c$ for all sufficiently large $n$.
\end{conj}
\noindent It would also be very interesting to improve the corresponding gap for multicolor Ramsey numbers; as discussed in \cref{sec:multicolor}, the bounds for $q \geq 3$ are very far apart.

Additionally, the structure of the counterexamples to \cref{conj:cfs} is somewhat dissatisfying, since $H_{k,n}$ consists of a small clique and an enormous number of isolated vertices. One can avoid using isolated vertices (for example, essentially the same proof\footnote{The only material difference is that now, one needs to replace \cref{lem:H-ub} by the weaker bound $r(H_{n,k}') \leq 3n$ (assuming $n \geq 4^k$). This bound is proved by repeatedly applying the Erd\H os--Szekeres bound to find disjoint monochromatic $K_k$, continuing as long as at least $n$ vertices remain. One can thus find $2n/k$ disjoint monochromatic $K_k$, half of which must have the same color, yielding a monochromatic $H_{n,k}'$.} goes through if one lets $H_{k,n}'$ be the disjoint union of $n/k$ copies of $K_k$), but there is a limit to how far one can push this using the techniques of this paper. Namely, we do not know how to prove a result like \cref{thm:main} unless the graph $H$ consists of a large number of small connected components. In particular, it seems that the following conjecture would require new ideas.
\begin{conj}
	There exists a graph $G$ and a vertex $v \in V(G)$ such that $H=G \setminus \{v\}$ is connected, and $r(G) = \omega(r(H))$. 
\end{conj}

It would also be interesting to study an ``average-case'' version of this question, rather than the ``worst-case'' version considered in \cref{conj:cfs,thm:main,thm:sqrt-ub}. For a graph $G$, its \emph{deck} $D(G)$ consists of all vertex-deleted induced subgraphs of $G$ (counted with multiplicity). \cref{conj:cfs} asks whether $r(H) = \Omega(r(G))$ for all $H \in D(G)$, and \cref{thm:main} says that this is false. However, rather than \emph{all}, one can ask for \emph{most}.
\begin{conj}
	There exists an absolute constant $c>0$ so that for all graphs $G$, at least $\frac 12 \ab{D(G)}$ of the graphs $H \in D(G)$ satisfy $r(H) \geq c\cdot r(G)$.
\end{conj}
In fact, it seems possible that $r(H) \geq c\cdot r(G)$ holds for all but $o(\ab{D(G)})$ of the graphs in $D(G)$. If this holds with appropriate control on the little-$o$, then it suffices for the original application of Conlon, Fox, and Sudakov; namely, such a result would show that $\log r(G(n,p))$ is concentrated in an interval of length $O(\sqrt n)$, by mimicking the proof of \cite[Theorem 5.6]{MR4170438}.

Another open problem concerns unbalanced colorings. Recall that the Erd\H os--Szemer\'edi theorem, \cref{lem:erdos-szemeredi}, says that unbalanced colorings contain much larger monochromatic cliques than what is implied by the Ramsey number $r(K_k)$ alone. We used this in the proof of \eqref{eq:second-bd} in \cref{prop:degen-balance}, where we found a monochromatic copy of $G$ by finding a sufficiently large monochromatic clique. This suggests a technique for strengthening \cref{prop:ramsey-balanced}, which is interesting in its own right: for graphs other than cliques, can one use the assumption of an unbalanced coloring to prove stronger Ramsey bounds?

Finally, rather than asking about vertex deletion, one could ask about edge deletion: if $H$ is obtained from $G$ by deleting a single edge, how large can $r(G)/r(H)$ be? It does not seem as though the techniques of \cref{sec:main-proof} can be used to construct examples where $r(G)/r(H)$ is super-constant, simply because the deletion of a single edge cannot split a connected graph into more than two connected components. In fact, we conjecture that there is no case where this ratio is super-constant.
\begin{conj}
	Let $G$ be a graph, and let $H$ be obtained from $G$ by deleting a single edge. Then
	\[
		r(H) \geq c \cdot r(G)
	\]
	for some absolute constant $c>0$.
\end{conj}

\paragraph{Acknowledgments.} I am indebted to Jacob Fox for many valuable discussions about this paper, and especially for his help with the proof of \cref{thm:sqrt-ub}. I would also like to thank David Conlon and the anonymous referees for many helpful comments on earlier drafts of this paper. Finally, I would like to thank Mehtaab Sawhney for suggesting the particularly short proof of \cref{lem:reverse jensen}, and am grateful to Nina Kam\v cev and Alp M\"uyesser for their permission to include their proof of \cref{cor:common-nbhd}.


\begin{thebibliography}{10}
\providecommand{\url}[1]{\texttt{#1}}
\providecommand{\urlprefix}{URL }
\providecommand{\eprint}[2][]{\url{#2}}

\bibitem{MR4267031}
Y.~Caro, A.~Hansberg, and A.~Montejano, Unavoidable chromatic patterns in
  2-colorings of the complete graph, \emph{J. Graph Theory} \textbf{97} (2021),
  123--147.

\bibitem{MR4186575}
D.~Conlon and A.~Ferber, Lower bounds for multicolor {R}amsey numbers,
  \emph{Adv. Math.} \textbf{378} (2021), Paper No. 107528, 5pp.

\bibitem{MR4170438}
D.~Conlon, J.~Fox, and B.~Sudakov, Short proofs of some extremal results {III},
  \emph{Random Structures Algorithms} \textbf{57} (2020), 958--982.

\bibitem{MR2433767}
J.~Cutler and B.~Mont\'{a}gh, Unavoidable subgraphs of colored graphs,
  \emph{Discrete Math.} \textbf{308} (2008), 4396--4413.

\bibitem{MR1556929}
P.~Erd\H{o}s and G.~Szekeres, A combinatorial problem in geometry,
  \emph{Compositio Math.} \textbf{2} (1935), 463--470.

\bibitem{MR325446}
P.~Erd\H{o}s and A.~Szemer\'{e}di, On a {R}amsey type theorem, \emph{Period.
  Math. Hungar.} \textbf{2} (1972), 295--299.

\bibitem{MR2455594}
J.~Fox and B.~Sudakov, Unavoidable patterns, \emph{J. Combin. Theory Ser. A}
  \textbf{115} (2008), 1561--1569.

\bibitem{MR2520279}
J.~Fox and B.~Sudakov, Density theorems for bipartite graphs and related
  {R}amsey-type results, \emph{Combinatorica} \textbf{29} (2009), 153--196.

\bibitem{MR4625871}
N.~Kam\v{c}ev and A.~M\"{u}yesser, Unavoidable patterns in locally balanced
  colourings, \emph{Combin. Probab. Comput.} \textbf{32} (2023), 796--808.

\bibitem{MR932230}
H.~Lefmann, A note on {R}amsey numbers, \emph{Studia Sci. Math. Hungar.}
  \textbf{22} (1987), 445--446.

\bibitem{MR4357352}
W.~Sawin, An improved lower bound for multicolor {R}amsey numbers and a problem
  of {Erd\H{o}s}, \emph{J. Combin. Theory Ser. A} \textbf{188} (2022), Paper
  No. 105579, 11pp.

\bibitem{MR4246789}
Y.~Wigderson, An improved lower bound on multicolor {R}amsey numbers,
  \emph{Proc. Amer. Math. Soc.} \textbf{149} (2021), 2371--2374.

\end{thebibliography}

\end{document}